\documentclass[11pt, a4paper]{article}
\usepackage{amsmath,amsfonts}   
\usepackage{amssymb,amsthm}     
\usepackage[top= 2.4cm, bottom=2.4 cm,left=2.4 cm,right=2.4 cm]{geometry} 
\mathversion{normal}
\usepackage{wrapfig}
\usepackage{graphicx,color}
\usepackage{pgf,tikz}
\usepackage{mathrsfs}
\usepackage[latin1]{inputenc}
\usepackage[english]{babel}
\usepackage{indentfirst}
\usepackage{amstext}
\usepackage{enumerate}
\usepackage{amsfonts}
\usepackage{textcomp}
\usetikzlibrary{arrows}
\usepackage{pgf,tikz,pgfplots}
\pgfplotsset{compat=1.14}
\usepackage{mathrsfs}

\newcommand{\N}{\mathbb{N}}

\newcommand{\G}{\mathcal{G}}
\newcommand{\GG}{\mathcal{G}}

\newcommand {\demo}{\hskip -0.6cm{\bf Proof.  }}
\newcommand {\fim}{\hfill{$\square$}\vskip 1pc}
\newtheorem{teo}{Theorem}[section]

\newtheorem{lema}[teo]{Lemma}
\newtheorem{prop}[teo]{Proposition}

\theoremstyle{definition}
\newtheorem{defin}[teo]{Definition}
\newtheorem{exe}[teo]{Example}
\newtheorem{obs}[teo]{Observation}
\newtheorem{remark}[teo]{Remark}
\begin{document}

\title{Li-Yorke chaos for ultragraph shift spaces}

\author{
\small{Daniel Gon\c{c}alves}\\
\footnotesize{UFSC -- Departmento de Matem\'atica}\\
\footnotesize{88040-900 Florian\'{o}polis - SC, Brazil}\\
\footnotesize{\texttt{daemig@gmail.com}}
\and
\small{Bruno Brogni Uggioni}\\
\footnotesize{IFRS -- Campus Canoas}\\
\footnotesize{92412-240 Canoas - RS, Brazil}, \\
\footnotesize{\texttt{bruno.uggioni@canoas.ifrs.edu.br}}}

\date{}

\maketitle

\begin{abstract} Recently, in connection with C*-algebra theory, the first author and Danilo Royer introduced ultragraph shift spaces. In this paper we define a family of metrics for the topology in such spaces, and use these metrics to study the existence of chaos in the shift. In particular we characterize all ultragraph shift spaces that have Li-Yorke chaos (an uncountable scrambled set), and prove that such property implies the existence of a perfect and scrambled set in the ultragraph shift space. Furthermore, this scrambled set can be chosen compact, what is not the case for a labelled edge shift (with the product topology) of an infinite graph. 
\end{abstract}

\vspace{1.0pc}

\vspace{1.0pc}

\vspace{1 cm}
\section{Introduction}

Symbolic dynamics over infinite alphabets is a subject of many interpretations, as there are many approaches to the definition of a shift space over an infinite alphabet. The most common approach so far is to take the shift space as the product space of the alphabet with the product topology, even though the resulting shift space is not even locally compact. This complete lack of compactness leads to some properties that differ greatly from  those of shift spaces over finite alphabets (for example, the entropy of a subshift may increase, see \cite{Petersen, Sal}). Recently, Ott-Tomforde-Willis proposed a shift space connected with C*-algebras and many of its properties were studied (see \cite{GR, GRultra, GSS, GSS1, OTW}). Deepening the connection with C*-algebras, and building on work of Webster (see \cite{Webster}), a new generalization of shifts of finite type to the infinite alphabet case was proposed in \cite{GRultrapartial} (see \cite{CG} for further connections with C*-algebras). The definition in \cite{GRultrapartial} relies on ultragraphs and the resulting shift space contains a countable basis of clopen subsets (which for ultragraphs that satisfy a mild condition turn to be compact-open subsets). Given the above setting, to understand, and compare, the dynamics of shifts over infinite alphabets is a relevant problem (note that Curtis-Hend-Lund type theorems for shifts over infinite alphabets were described in \cite{GSCSC, GSS0, GS}). In this paper we study Li-Yorke Chaos for the ultragraph shift spaces defined in \cite{GRultrapartial}. In particular, regarding Li-Yorke chaoticity, we show that these shifts behave like shifts over a finite alphabet, what is not the case for shift spaces defined with the product topology (see \cite{Raines}). 

There are several notions of chaos in the mathematical literature, as one can see for instance in \cite{Kolyada}, where the author presents a brief survey of the concepts of chaos and relates them with topological properties of the associated systems. Informally, one can say that the basic idea present in many approaches is the following: there exists chaos when one can not predict the behavior of many trajectories of a given system, even in the case when it is possible to intuit the location of some points of the trajectory. Historically speaking, the first definition of chaos in a dynamical system was given by Li and Yorke in \cite{Li-Yorke}: given a map $f: X \rightarrow X$ between the compact metric space $X$, and $x,y \in X$ any distinct elements, the pair $(x,y)$ is called \textit{scrambled} if:

\begin{center}
(i) $\displaystyle{\limsup_{n \rightarrow \infty}d(f^{n}(x),f^{n}(y)) > 0}$,\\
(ii) $\displaystyle{\liminf_{n \rightarrow \infty}d(f^{n}(x),f^{n}(y)) = 0}$.\\
\end{center}
More generally, a set $A \subseteq X$ is \textit{scrambled} if every pair $x,y \in A$ is scrambled, and the system $(X,f,d)$ is \textit{chaotic} (or satisfies \textit{Li-Yorke chaos}) if it admits an uncountable scrambled set. 

The context of compact metric spaces equipped with a continuous function is probably the most studied context regarding presence of chaos. For instance, in \cite{Ethan} and \cite{Kolyada} the authors discuss several notions of chaos and sensitivity. Also, in \cite {Garcia} the notion of mean Li-Yorke chaos is found and in \cite{DaiX} chaos is studied in a more general setting. Recently the study of Li-Yorke and other types of chaos has been expanded to the non compact and measurable setting. For the measure theory context see \cite{Downarowicz}, and, more related to our purposes, Raines and Underwood, in \cite{Raines}, focus on shifts (with product topology) associated to an infinite countable alphabet. More precisely, in \cite{Raines} it is shown that the edge-shift space (with the product topology) associated with an infinite graph has Li-Yorke chaos if, and only if, it admits a single scrambled pair, and, even more, in this case the scrambled set may be chosen closed and perfect (but not necessarily compact). Furthermore, Raines and Underwood present an example of a sofic shift over an infinite alphabet which admits a scrambled pair (therefore infinitely many scrambled pairs) but does not satisfy Li-Yorke chaos. 

Motivated by the above, we investigate the presence of Li-Yorke chaos in ultragraph shift spaces. We notice some important differences and resemblances between the uncountable scrambled sets of ultragraph shift spaces and of edge shift spaces with product topology. In our case, for instance, the scrambled set can be chosen compact and perfect (Theorem \ref{teothebest}), in the same way that it can be done for shifts of finite type over finite alphabets (see \cite{Raines}). On the other hand, as it happens for the labelled edge shifts studied in \cite{Raines}, we are able to find an ultragraph shift space which admits a scrambled pair but does not present Li-Yorke chaos.
 
We organize the paper as follows. In Section~2 we set up basic notation and recall some relevant results from the literature regarding graphs and ultragraphs. Also, we define a family of metrics in the ultragraph shift space $X$ associated to an ultragraph $\mathcal{G}$ (this is a new result). In Section~3 we describe Li-Yorke chaoticity in ultragraph shift spaces by showing, among other results, that the existence of a scrambled pair formed by infinite paths implies Li-Yorke chaos, that is, the existence of an uncountable scrambled set. Even more, the uncountable scrambled set can always be chosen compact. We also present an example of an ultragraph (which is not a graph) such that the associated shift space admits a scrambled pair but does not present Li-Yorke chaos. We summarize our results related to chaos in ultragraph shift spaces in Theorem~\ref{teosumma}.

\section{Ultragraph shift spaces}

We begin with some standard definitions regarding ultragraphs, as introduced in \cite{Marrero} and \cite{T}. After this we recall the definition of an ultragraph shift space $X$, as given in \cite{GRultrapartial}, and then define a family of metrics for the topology in $X$.

\subsection{Background}

\begin{defin}\label{def of ultragraph}
An \emph{ultragraph} is a quadruple $\mathcal{G}=(G^0, \mathcal{G}^1, r,s)$ consisting of two countable sets $G^0, \mathcal{G}^1$, a map $s:\mathcal{G}^1 \to G^0$, and a map $r:\mathcal{G}^1 \to P(G^0)\setminus \{\emptyset\}$, where $P(G^0)$ stands for the power set of $G^0$.
\end{defin}

\begin{defin}\label{def of mathcal{G}^0}
Let $\mathcal{G}$ be an ultragraph. Define $\mathcal{G}^0$ to be the smallest subset of $P(G^0)$ that contains $\{v\}$ for all $v\in G^0$, contains $r(e)$ for all $e\in \mathcal{G}^1$, and is closed under finite unions and nonempty finite intersections.
\end{defin}



Let $\mathcal{G}$ be an ultragraph. A \textit{finite path} in $\mathcal{G}$ is either an element of $\mathcal{G}%
^{0}$ or a sequence of edges $e_{1}\ldots e_{k}$ in $\mathcal{G}^{1}$ where
$s\left(  e_{i+1}\right)  \in r\left(  e_{i}\right)  $ for $1\leq i\leq k$. If
we write $\alpha=e_{1}\ldots e_{k}$, the length $\left|  \alpha\right|  $ of
$\alpha$ is just $k$. The length $|A|$ of a path $A\in\mathcal{G}^{0}$ is
zero. We define $r\left(  \alpha\right)  =r\left(  e_{k}\right)  $ and
$s\left(  \alpha\right)  =s\left(  e_{1}\right)  $. For $A\in\mathcal{G}^{0}$,
we set $r\left(  A\right)  =A=s\left(  A\right)  $. The set of
finite paths in $\mathcal{G}$ is denoted by $\mathcal{G}^{\ast}$. An \textit{infinite path} in $\mathcal{G}$ is an infinite sequence of edges $\gamma=e_{1}e_{2}\ldots$ in $\prod \mathcal{G}^{1}$, where
$s\left(  e_{i+1}\right)  \in r\left(  e_{i}\right)  $ for all $i$. The set of
infinite paths  in $\mathcal{G}$ is denoted by $\mathfrak
{p}^{\infty}$. The length $\left|  \gamma\right|  $ of $\gamma\in\mathfrak
{p}^{\infty}$ is defined to be $\infty$. A vertex $v$ in $\mathcal{G}$ is
called a \emph{sink} if $\left|  s^{-1}\left(  v\right)  \right|  =0$ and is
called an \emph{infinite emitter} if $\left|  s^{-1}\left(  v\right)  \right|
=\infty$. 

For $n\geq1,$ we define
$\mathfrak{p}^{n}:=\{\left(  \alpha,A\right)  :\alpha\in\mathcal{G}^{\ast
},\left\vert \alpha\right\vert =n,$ $A\in\mathcal{G}^{0},A\subseteq r\left(
\alpha\right)  \}$. We specify that $\left(  \alpha,A\right)  =(\beta,B)$ if
and only if $\alpha=\beta$ and $A=B$. We set $\mathfrak{p}^{0}:=\mathcal{G}%
^{0}$ and we let $\mathfrak{p}:=\coprod\limits_{n\geq0}\mathfrak{p}^{n}$. We embed the set of finite paths $\GG^*$ in $\mathfrak{p}$ by sending $\alpha$ to $(\alpha, r(\alpha))$. We
define the length of a pair $\left(  \alpha,A\right)  $, $\left\vert \left(
\alpha,A\right)  \right\vert $, to be the length of $\alpha$, $\left\vert
\alpha\right\vert $. We call $\mathfrak{p}$ the \emph{ultrapath space}
associated with $\mathcal{G}$ and the elements of $\mathfrak{p}$ are called
\emph{ultrapaths} (or just paths when the context is clear). Each $A\in\mathcal{G}^{0}$ is regarded as an ultrapath of length zero and can be identified with the pair $(A,A)$. We may extend the range map $r$ and the source map $s$ to
$\mathfrak{p}$ by the formulas, $r\left(  \left(  \alpha,A\right)  \right)
=A$, $s\left(  \left(  \alpha,A\right)  \right)  =s\left(  \alpha\right)
$ and $r\left(  A\right)  =s\left(  A\right)  =A$.

We concatenate elements in $\mathfrak{p}$ in the following way: If $x=(\alpha,A)$ and $y=(\beta,B)$, with $|x|\geq 1, |y|\geq 1$, then $x\cdot y$ is defined if and only if
$s(\beta)\in A$, and in this case, $x\cdot y:=(\alpha\beta,B)$. Also we
specify that:
\begin{equation}
x\cdot y=\left\{
\begin{array}
[c]{ll}%
x\cap y & \text{if }x,y\in\mathcal{G}^{0}\text{ and if }x\cap y\neq\emptyset\\
y & \text{if }x\in\mathcal{G}^{0}\text{, }\left|  y\right|  \geq1\text{, and
if }x\cap s\left(  y\right)  \neq\emptyset\\
x_{y} & \text{if }y\in\mathcal{G}^{0}\text{, }\left|  x\right|  \geq1\text{,
and if }r\left(  x\right)  \cap y\neq\emptyset
\end{array}
\right.  \label{specify}%
\end{equation}
where, if $x=\left(  \alpha,A\right)  $, $\left|  \alpha\right|  \geq1$ and if
$y\in\mathcal{G}^{0}$, the expression $x_{y}$ is defined to be $\left(
\alpha,A\cap y\right)  $. Given $x,y\in\mathfrak{p}$, we say that $x$ has $y$ as an \emph{initial segment} if
$x=y\cdot x^{\prime}$, for some $x^{\prime}\in\mathfrak{p}$, with $s\left(
x^{\prime}\right)  \cap r\left(  y\right)  \neq\emptyset$. 

We extend the source map $s$ to $\mathfrak
{p}^{\infty}$, by defining $s(\gamma)=s\left(  e_{1}\right)  $, where
$\gamma=e_{1}e_{2}\ldots$. We may concatenate pairs in $\mathfrak{p}$, with
infinite paths in $\mathfrak{p}^{\infty}$ as follows. If $y=\left(
\alpha,A\right)  \in\mathfrak{p}$, and if $\gamma=e_{1}e_{2}\ldots\in
\mathfrak{p}^{\infty}$ are such that $s\left(  \gamma\right)  \in r\left(
y\right)  =A$, then the expression $y\cdot\gamma$ is defined to be
$\alpha\gamma=\alpha e_{1}e_{2}...\in\mathfrak{p}^{\infty}$. If $y=$
$A\in\mathcal{G}^{0}$, we define $y\cdot\gamma=A\cdot\gamma=\gamma$ whenever
$s\left(  \gamma\right)  \in A$. Of course $y\cdot\gamma$ is not defined if
$s\left(  \gamma\right)  \notin r\left(  y\right)  =A$. 


\begin{defin}
\label{infinte emitter} For each subset $A$ of $G^{0}$, let
$\varepsilon\left(  A\right)  $ be the set $\{ e\in\mathcal{G}^{1}:s\left(
e\right)  \in A\}$. We shall say that a set $A$ in $\mathcal{G}^{0}$ is an
\emph{infinite emitter} whenever $\varepsilon\left(  A\right)  $ is infinite.
\end{defin}

The key concept in the definition of the shift space $X$ associated to $\GG$ is that of minimal infinite emitters. We recall this below.

\begin{defin}\label{minimal} Let $\GG$ be an ultragraph and $A\in \GG^0$. We say that $A$ is a minimal infinite emitter if it is an infinite emitter that contains no proper subsets (in $\GG^0$) that are infinite emitters. Equivalently, $A$ is a minimal infinite emitter if it is an infinite emitter and has the property that, if $B\in \GG^0$ is an infinite emitter, and $B\subseteq A$, then $B=A$. For a finite path $\alpha$ in $\GG$, we say that $A$ is a minimal infinite emitter in $r(\alpha)$ if $A$ is a minimal infinite emitter and $A\subseteq r(\alpha)$. We denote the set of all minimal infinite emitters in $r(\alpha)$ by $M_\alpha$.
\end{defin}

For later use we recall the following Lemma.

\begin{lema}\label{miniftyem} Let $x=(\alpha, A) \in \mathfrak{p}$ and suppose that $A$ is a minimal infinite emitter. If the cardinality of $A$ is finite then it is equal to one and, if the cardinality of $A$ is infinite, then $A = \bigcap\limits_{e\in Y} r(e)$ for some finite set $Y\subseteq \GG^1$.
\end{lema}



Associated to an ultragraph with no sinks, we have the topological space $X= \mathfrak{p}^{\infty} \cup X_{fin}$, where 
$$X_{fin} = \{(\alpha,A)\in \mathfrak{p}: |\alpha|\geq 1 \text{ and } A\in M_\alpha \}\cup
 \{(A,A)\in \GG^0: A \text{ is a minimal infinite emitter}\}, $$ and the topology has a basis given by the collection $$\{D_{(\beta,B)}: (\beta,B) \in \mathfrak{p}, |\beta|\geq 1\ \} \cup \{D_{(\beta, B),F}:(\beta, B) \in X_{fin}, F\subset \varepsilon\left( B \right), |F|<\infty \},$$ where for each $(\beta,B)\in \mathfrak{p}$ we have that $$D_{(\beta,B)}= \{(\beta, A): A\subset B \text{ and } A\in M_\beta \}\cup\{y \in X: y = \beta \gamma', s(\gamma')\in B\},$$ and, for $(\beta,B)\in X_{fin}$ and $F$ a finite subset of $\varepsilon\left( B \right)$,  $$D_{(\beta, B),F}=  \{(\beta, B)\}\cup\{y \in X: y = \beta \gamma', \gamma_1' \in \ \varepsilon\left( B \right)\setminus F\}.$$
 
\begin{remark}\label{cylindersets} For every $(\beta,B)\in \mathfrak{p}$, we identify $D_{(\beta, B)}$ with $D_{(\beta, B),F}$, where $F=\emptyset$. Furtheremore, we call the basic elements of the topology of $X$ given above by \emph{cylinder sets}.
\end{remark}

For our work the description of convergence of sequences in $X$ is important, so we recall it below.

\begin{prop}\label{convseq} Let $(x^n)_{n=1}^{\infty}$ be a sequence of elements in $X$, where $x^n = (\gamma^n_1\ldots \gamma^n_{k_n}, A_n)$ or $x^n = \gamma_1^n \gamma_2^n \ldots$, and let $x \in X$.
\begin{enumerate}[(a)]
\item If $|x|= \infty$, say $x=\gamma_1 \gamma_2 \ldots$, then $\{x^n\}_{n=1}^{\infty}$ converges to $x$ if, and only if, for every $M\in \N$ there exists $N\in \N$ such that $n>N$ implies that $|x^n|\geq M$ and $\gamma^n_i= \gamma_i$ for all $1\leq i \leq M$.
\item If $|x|< \infty$, say $x=(\gamma_1 \ldots \gamma_k, A)$, then $\{x^n\}_{n=1}^{\infty}$ converges to $x$ if, and only if, for every finite subset $F\subseteq \varepsilon\left(  A\right)$ there exists $N\in \N$ such that $n > N$ implies that $x^n = x$ or $|x^n|> |x|$, $\gamma^n_{|x|+1} \in \ \varepsilon\left(  A\right)\setminus F$, and $\gamma^n_i = \gamma_i$ for all $1 \leq i \leq |x|$. 
\end{enumerate}
\end{prop}

Finally, attached to the space $X$ we have the shift map, as in \cite{GRultrapartial}:

\begin{defin}\label{shift-map-def}
The \emph{shift map} is the function $\sigma : X \rightarrow X$ defined by $$\sigma(x) =  \begin{cases} \gamma_2 \gamma_3 \ldots & \text{ if $x = \gamma_1 \gamma_2 \ldots \in \mathfrak{p}^\infty$} \\ (\gamma_2 \ldots \gamma_n,A) & \text{ if $x = (\gamma_1 \ldots \gamma_n,A) \in X_{fin}$ and $|x|> 1$} \\(A,A) & \text{ if $x = (\gamma_1,A) \in X_{fin}$} \\ (A,A) & \text{ if $x = (A,A)\in X_{fin}$.} 
\end{cases}$$
\end{defin}

\subsection{A family of metrics in the ultragraph shift space}

It was shown in \cite{GRultrapartial} that the ultragraph shift space associated to an ultragraph is metrizable, but a specific metric was not built. In what follows, building from ideas presented in \cite{OTW} and \cite{Webster}, we define a family of metrics that induce the topology in $X$. We do so by embedding $X$ into $2^{\mathfrak{p}} = \{ 0, 1 \}^{\mathfrak{p}}$ (where we consider $\{ 0, 1 \}^{\mathfrak{p}}$ with the product topology), and then we use the embedding to ``pull back'' the metric in $2^{\mathfrak{p}}$ to a metric on $X$. 

\begin{defin}\label{defalpha}
We define a function $\alpha : X \rightarrow \{ 0, 1 \}^{\mathfrak{p}}$ by $$\alpha (x) (y) = \begin{cases} 1 & \text{ if $y$ is an initial segment of $x$,} \\ 0 & \text{ otherwise.} \end{cases}$$
\end{defin}

\begin{remark} Notice that convergence in $\{ 0, 1 \}^{\mathfrak{p}}$ is point-wise convergence.
\end{remark}

\begin{lema}\label{lemmahomeo}
The map $\alpha$ defined above is an embedding of $X$ into $\{ 0, 1 \}^{\mathfrak{p}}$, that is, it is a homeomorphism between $X$ and $\alpha(X)$.
\end{lema}
\demo

Clearly $\alpha$ is injective. We will show that it is bi-continuous.

Let $(x^i)$ be a sequence in $X$, where $x^i = (\gamma^i_1\ldots \gamma^i_{k_i}, A_i)$ or $x^i = \gamma_1^i \gamma_2^i \ldots$, converging to $x \in X$. We will show that, for every $y\in \mathfrak{p}$, the sequence $\alpha(x^i)(y)$ converges to $\alpha(x)(y)$.

Suppose that $|x|=\infty$, say $x=\gamma_1 \gamma_2 \ldots$. Let $p \in\mathfrak{p}$. Since $(x^i)$ converges to $x$, there exists $N>0$ such that, for all $i>N$, $\gamma^i_j = \gamma_j$ for all $j=1, \ldots, |p|+1$. Hence $\alpha(x^i)(p) = \alpha (x)(p)$ for all $i>N$.

Now suppose that  $|x|< \infty$, say $x=(\gamma_1 \ldots \gamma_k, A)$. Let $p \in\mathfrak{p}$, say $$p=(\beta_1 \ldots \beta_l, B).$$ 
If $|p|>|x|=k$ then $\alpha(x)(p)=0$. Since $(x^i)$ converges to $x$ there exists $N>0$ such that, for all $i>N$, we have $x^i = x$ or $\gamma_{k+1}^i \in \varepsilon(A)\setminus \{\beta_{k+1}\}$. Hence $\alpha(x^i)(p) = 0$ for all $i>N$. The case $|p|<|x|$ is straightforward. So we focus on the case $|p| = |x|$ (in this case $k=l$). We divide this in three sub-cases:

\begin{enumerate}[(a)]
\item $\beta_1 \ldots \beta_k \neq \gamma_1 \ldots \gamma_k$. This is straightforward.
\item $\beta_1 \ldots \beta_k = \gamma_1 \ldots \gamma_k$ and $ A \subseteq B$. Then $\alpha(x)(p)=1$. Let $N>0$ be such that for all $i> N$ it holds that $x^i = x$ or $\gamma^i_j= \gamma_j$, for all $j=1, \ldots, k$ and $s(\gamma^i_{k+1})\in A$. Then for all $i> N$ we have that $\alpha(x^i)(p)=1$.
\item $\beta_1 \ldots \beta_k = \gamma_1 \ldots \gamma_k$, and there exists a vertex $v\in A\setminus B$. Then $\alpha(x)(p)=0$. Since $A$ is a minimal infinite emitter, by Lemma~\ref{miniftyem} either $A=\{v\}$ or the cardinality of $A$ is infinite. Suppose that $A=\{v\}$. By the convergence of $(x^i)$ there exists $N>0$ such that, for all $i>N$ $x^i=x$ or $s(\gamma^i_{k+1})=v$. Hence $\alpha(x^i)(p)=0$ for all $i>N$. Finally, suppose that the cardinality of $A$ is infinite. Then the cardinality of $A\cap B$ is finite and $|s^{-1}(v)|<\infty$ for all $v\in A\cap B$ (since $A$ is a minimal infinite emitter). Let $F= \varepsilon(B\cap A)$. Then, by the convergence of $(x^i)$, there exists $N>0$ such that, for all $i>N$ either $x^i=x$ or $\gamma^i_{k+1} \in \varepsilon(A)\setminus F$. Hence $\alpha(x^i)(p)=0$ for all $i>N$ as desired.
\end{enumerate}

Since our space is not compact we can not infer that $\alpha^{-1}$ is continuous and have to prove it directly. We do this below.

Let $U$ be an open set in $X$. We have to show that $\alpha(U)$ is open. Let $y\in \alpha (U)$. Then $y = \alpha (x)$ for some $x\in U$. Suppose that $|x|= \infty$, say $x=\gamma_1 \gamma_2 \ldots$. Since $U$ is open, there is a basis element $D_{(\gamma_1\ldots \gamma_k)} $ such that $x\in D_{(\gamma_1\ldots \gamma_k)}\subseteq U$. Let $s = (\gamma_1 \ldots \gamma_k, \{s(\gamma_{k+1})\})$ and consider the open neighborhood $B=\{1\}_s \times \{0,1\}\times \{0,1\} \times \ldots$. Then $B\cap \alpha(X) \subseteq \alpha (U)$ since if $y=\alpha(x_1)$ in $B$ then $s$ is an initial segment of $x_1$. Now suppose that $|x|<\infty$, say $x=(\gamma_1 \ldots \gamma_k, B)$. Again, since $U$ is open, there is a basis element $D_{x,F}$ such that $x\in D_{x,F}\subseteq U$. Let s=$(\gamma_1 \ldots \gamma_k, \gamma_{k+1}, \{r(\gamma_{k+1})\})$ and consider the open neighborhood $B=\{1\}_s \times \{0,1\}\times \{0,1\} \times \ldots$. Then, as before, $B\cap \alpha(X) \subseteq \alpha (U)$ since if $y=\alpha(x_1)$ in $B$ then $s$ is an initial segment of $x_1$, and hence $x_1 \in D_{x,F}\subseteq U$.

\fim

Notice that  $\mathfrak{p}$ is countable. Thus we may list the elements of $\mathfrak{p}$ as $\mathfrak{p} = \{ p_1, p_2, p_3, \ldots \}$, order $\{ 0, 1 \}^{\mathfrak{p}}$ as $$\{ 0, 1 \}^{\mathfrak{p}} = \{0, 1 \}_{p_1} \times \{0,1\}_{p_2} \times \ldots,$$
and define a metric $d_\textrm{fin}$ on $\{ 0, 1 \}^{\mathfrak{p}}$ by 
$$
d_\textrm{fin} (\mu, \nu) := \begin{cases} 1/2^i & \text{$i \in \N$ is the smallest value such that $\mu(i) \neq \nu(i)$} \\
0 & \text{if $\mu(i) = \nu (i)$ for all $i \in \N$}.
\end{cases}$$
The metric $d_\textrm{fin}$ induces the product topology on $\{ 0, 1 \}^{\mathfrak{p}}$, and hence the topology on $X$ is induced by the metric $d_X$ on $X$ defined by $d_X(x,y) := d_\textrm{fin}( \alpha(x), \alpha(y))$.  Note that for $x,y \in X$, we have
\begin{equation}\label{definmetricanova}
d_X (x,y) := \begin{cases} 1/2^i & \text{$i \in \N$ is the smallest value such that $p_i$ is an initial} \\  & \text{ \ \ \ segment of one of $x$ or $y$ but not the other,} \\
0 & \text{if $x=y$}.
\end{cases}
\end{equation} 

\begin{remark}\label{ordem}
Notice that the metric $d_X$ depends on the order we choose for $\mathfrak{p} = \{ p_1, p_2, p_3, \ldots \}$.
\end{remark}

To illustrate how to use the metric we just defined, next we give a proof that, depending on the shift space $X$, the shift map $\sigma : X \rightarrow X$ is not continuous at elements of length zero. For instance, let $\{\gamma_{n}\}_{n \in \mathbb{N}}$ be a sequence of pair-wise distinct edges in a graph with only one vertex, and consider the sequences given by $x^{n} := \gamma_{n}\gamma_{1}\gamma_{1}\gamma_{1}\gamma_{1}\ldots$ and $y^{n} := \gamma_{n+1}\gamma_{2}\gamma_{2}\gamma_{2}\ldots$. As $\gamma_{j} \neq \gamma_{\ell}$ for all naturals $j \neq \ell$, if we fix $n$, then $\gamma_{n}$ is an initial segment of $x^{n}$ which is not an initial segment of $y^{n}$. On the same way, $\gamma_{n+1}$ is an initial segment of $y^{n}$ which is not an initial segment of $x^{n}$. So, there is a pair-wise distinct set of segments $\{p_{j_{n}}: n \in \mathbb{N}\}$ (see equality (\ref{definmetricanova}) and Remark (\ref{ordem})) such that $\gamma_{n}$ (or $\gamma_{n+1}$) is an initial segment of $p_{j_{n}}$ and satisfies: 
$$d_{X}(x^{n},y^{n}) = \dfrac{1}{2^{j_{n}}} \rightarrow 0.$$
On the other hand, for all natural $n$:
$$d_{X}(\sigma(x^{n}),\sigma(y^{n})) \geq \max \left\{\dfrac{1}{2^{j_{1}}},\dfrac{1}{2^{j_{2}}}\right\} > 0,$$
and hence $\sigma$ is not continuous.

\section{Li-Yorke Chaos and ultragraphs}

In this section we describe Li-Yorke chaos in ultragraph shift spaces. As we want to study the presence of chaos in the system $(X,\sigma, d_{X})$, we need a definition of chaos that does not involve continuity. The next definition is based on the works of Li and Yorke (\cite{Li-Yorke}) and Raines (\cite{Raines}). The only difference is that, in our case, we do not ask for the transformation to be continuous.

\begin{defin}\label{definscramble}
Let $f: M \rightarrow M$ be a map of a metric space $M$ with metric $d$, and let $x,y \in M$. The pair $x,y$ is \textit{scrambled} if
\begin{enumerate}[(i)]
\item $\displaystyle{\limsup_{n \rightarrow \infty}d(f^{n}(x),f^{n}(y)) > 0}$, and
\item $\displaystyle{\liminf_{n \rightarrow \infty}d(f^{n}(x),f^{n}(y)) = 0}$.
\end{enumerate}
A set $A \subseteq M$ is \textit{scrambled} if every pair $x,y \in A$ is scrambled and the system $(M,f,d)$ is \textit{chaotic} (or satisfies \textit{Li-Yorke chaos}) if it admits an uncountable scrambled set. 
\end{defin}

Let $\G$ be an ultragraph and $X$ be the associated shift space. In the next four propositions we characterize the asymptotic behavior of pairs of elements in $X$. More precisely, we characterize when a pair satisfy one of the limit equalities in the definition of a scrambled pair, that is, we describe the pairs $(x,y)$ for which $\displaystyle{\limsup_{n \rightarrow \infty} d_{X}(\sigma^{n}(x),\sigma^{n}(x)) >0}$ or $\displaystyle{\liminf_{n \rightarrow \infty} d_{X}(\sigma^{n}(x),\sigma^{n}(y)) =0}$. 

Notice that if $x = (\gamma_{1}\gamma_{2} \ldots \gamma_{k}, A)$ and $y=(\beta_{1}\beta_{2}\ldots\beta_{\ell}, B)$ are two finite paths, with $A$ and $B$ minimal infinite emitters, then $x,y$ is not a scrambled pair, since we can find a natural $N$ such that $\sigma^{n}(x) = A$ and $\sigma^{n}(y) = B$ for all $n \geq N$. Therefore, when looking for scrambled pairs, we just need to consider two cases, namely when $x$ is an infinite path and $y$ is a minimal infinite emitter (see Propositions~\ref{teosupemm} and \ref{teoinfemm}) or when $x$ and $y$ are both infinite paths (see Propositions \ref{Teolimsup0} and \ref{Teoinf0}). Before we proceed we need the following auxiliary result.

\begin{lema}\label{lemata} Let $x = \gamma_{1}\gamma_{2}\ldots \in \mathfrak{p}^{\infty}$ be an infinite path, $A \in \mathfrak{p}^{0}$ be a minimal infinite emitter and let $\{n_k\}$ be a subsequence of the natural numbers such that $s(\sigma^{n_{k}}(x))\in A$ for all $k$. If for all natural $m$ at most finitely many $\gamma_{n_i}$ satisfy $\gamma_{n_i} = \gamma_{m}$ then $d_{X}(\sigma^{n_{k}}(x),A)  \rightarrow 0.$
\end{lema}
\begin{proof}
Let $x$, $A$, and $\{n_k\}$ be as in the statement of the proposition. Notice that $A$ is an initial segment of $\sigma^{n_{k}}(x)$ for all $k$. Hence the distance between $A$ and $\sigma^{n_{k}}(x)$ is determined by the position, in the enumeration of $\mathfrak{p}$, of a segment (call it $p_{n'_{k}}$)  that initiates at $\sigma^{n_{k}}(x)_1$, 
(see the definition of the metric (\ref{definmetricanova}) and Remark~(\ref{ordem})). Now, if for each $m$ at most finitely many $\gamma_{n_i}$ satisfy $\gamma_{n_i} = \gamma_{m}$, then the set $\{p_{n'_{k}}: k \in \mathbb{N}\}$ has infinite cardinality and the set $\{k: p_{n'_{k}}= p_{n'_{j}}\}$ has finite cardinality for all $j$. Therefore
$$d_{X}(\sigma^{n_{k}}(x),A) = \dfrac{1}{2^{n'_{k}}} \rightarrow 0.$$
\end{proof}

\begin{prop}\label{teosupemm}
Let $x = \gamma_{1}\gamma_{2}\ldots \in \mathfrak{p}^{\infty}$ be an infinite path and $A \in \mathfrak{p}^{0}$ be a minimal infinite emitter. Then
$$\displaystyle{\limsup_{n \rightarrow \infty} d_{X}(\sigma^{n}(x), A) >0}$$
if, and only if, one of the following two conditions hold: either there is an edge $\gamma_{m}$ and infinitely many indices $i_{k}$ such that $\gamma_{i_{k}} = \gamma_{m}$, or there are infinitely many edges $\gamma_{j_{k}}$, pair-wise different, such that $s(\gamma_{j_{k}}) \notin A$, for all $k \in \mathbb{N}$. 
\end{prop}
\begin{proof}
Let $x = \gamma_{1}\gamma_{2}\ldots$ be an infinite path and $A$ be a minimal infinite emitter.

We prove the converse of the proposition first. 
Suppose that there exists infinitely many indices $i_k$, $k=1,2,\ldots$, such that $\gamma_{i_{k}} = \gamma_{m}$. Then, for each $k$, the first entry of $\sigma^{i_{k}-1}(x)$ is $ \gamma_{m}$. But, $\gamma_{m}$ is not an initial segment of $A$. Therefore, if $\gamma_{m}$ is the element $p_{j}$ in the enumeration of $\mathfrak{p}$, we obtain that
$$d_{X}(\sigma^{i_{k}-1}(x),A) \geq \dfrac{1}{2^{j}},$$ for all $k=1,2\ldots .$
Hence $\displaystyle{\limsup_{n \rightarrow \infty} d_{X}(\sigma^{n}(x), A) >0}$.
Now, suppose that there are infinitely many entries of $x$, pair-wise different, say $\{\gamma_{j_{k}}\}_{k \in \mathbb{N}}$, such that $s(\gamma_{j_{k}}) \notin A$. Then $A$ is not an initial segment of $\gamma_{j_{k}}$, for all $k$. Therefore, if $A$ is the element $p_{i}$ in the enumeration of $\mathfrak{p}$, we have that 
$$d_{X}(\sigma^{j_{k}-1}(x),A) \geq \dfrac{1}{2^{i}},$$
and hence $\displaystyle{\limsup_{n \rightarrow \infty} d_{X}(\sigma^{n}(x), A) >0}$ as desired.

Next we show the forward implication of the proposition. For this assume that $$\displaystyle{\limsup_{n \rightarrow \infty} d_{X}(\sigma^{n}(x), A) >0}$$ and, for all natural $m$, at most finitely many $\gamma_{k}$ satisfy $\gamma_{k} = \gamma_{m}$.
Let $\{\sigma^{m_{k}}(x)\}_{k \in \mathbb{N}}$ be a sequence such that $$\displaystyle{d_{X}(\sigma^{m_{k}}(x),A) \rightarrow \limsup_{n \rightarrow \infty} d_{X}(\sigma^{n}(x), A) > 0}.$$ 
If there are infinitely many $k$ for which $s(\sigma^{m_{k}}(x)) \in A$, then we can pass to a subsequence of $\{\sigma^{m_{k}}(x)\}_{k \in \mathbb{N}}$ and apply Lemma~\ref{lemata} to obtain a contradiction. Therefore only a finite number of $k$ are such that $s(\sigma^{m_{k}}(x)) \in A$. We conclude that there are infinitely many entries $\gamma_{j_{k}}$, pair-wise different, such that $s(\gamma_{j_{k}}) \notin A$ for all $k \in \mathbb{N}$ (since by assumption for all natural $m$ at most finitely many $\gamma_{k}$ satisfy $\gamma_{k} = \gamma_{m}$). 

\end{proof}

In the next proposition we partially characterize the asymptotic behavior of pairs of infinite sequences.

\begin{prop}\label{Teolimsup0}
Let $x = \gamma_{1}\gamma_{2}\gamma_{3}\ldots$ and $y = \beta_{1}\beta_{2}\beta_{3} \ldots$ be infinite paths. Then
 $$\displaystyle{\limsup_{n \rightarrow \infty}d_{X}(\sigma^{n}(x), \sigma^{n}(y)) >0}$$
if, and only if, there is a subsequence $\{i_k\}$ of the naturals and $\gamma_m$ (or $\beta_m$) such that $\gamma_{i_{k}} = \gamma_{m}$ (or $\beta_{i_{k}}=\beta_m$) and $\gamma_{i_{k}} \neq \beta_{i_{k}}$ for all $k \in \mathbb{N}$.
\end{prop}

\begin{proof}
Before we begin the proof we make an easy observation: if $x'=x'_1x'_2\ldots$ and $y'=y'_1 y'_2\ldots$ are infinite paths such that $x'_{k} \neq y'_{k}$, for some natural $k$, then for any $ n \geq m \geq k$, the segment $x'_{1}x'_{2}\ldots x'_{m}$ is not an initial segment of $y'_{1}y'_{2}\ldots y'_{n}$ and vice versa. In particular, if $x'_{1} \neq y'_{1}$ then no initial segment of $x'$ is an initial segment of $y'$ and vice versa.

Now we prove the converse of the proposition. Suppose that $x$ and $y$ are infinite paths such that there are infinitely many edges $\gamma_{i_{k}}$ such that $\gamma_{i_{k}} = \gamma_{m}$ and $\gamma_{i_{k}} \neq \beta_{i_{k}}$ for all natural $k$. Then, for every natural $k$, the first entry of $\sigma^{i_{k}-1}(x)$ is equal $\gamma_m$ and different from the first entry of $\sigma^{i_{k}-1}(y)$. Assume that the edge $\gamma_{m}$ appears in the enumeration of $\mathfrak{p}$ as $p_{j}$. Then $$d_{X}(\sigma^{i_{k}-1}(x),\sigma^{i_{k}-1}(y))\geq \frac{1}{2^{j}}, \text{ for all }k \in \mathbb{N},$$
and hence $\displaystyle{\limsup_{ n \rightarrow \infty}d_{X}(\sigma^{n}(a), \sigma^{n}(b)) >0}$. The case with infinitely many $\beta_{i_{k}}$ equal to a $\beta_m$ is analogous.

For the forward implication, let $x=\gamma_1 \gamma_2 \ldots$ and $y=\beta_1 \beta_2 \ldots$ be infinite paths such that $\displaystyle{\limsup_{n \rightarrow \infty} d_{X}(\sigma^{n}(x), \sigma^{n}(y)) >0}$. Then, the set $$L = \{\ell: \gamma_{\ell} \neq \beta_{\ell}\}$$ is infinite. 
Assume that there is no constant subsequence of $x$ neither of $y$.  
As $L$ is infinite, for each natural $n$ there is an index $\ell_{n} \in L$, and a finite path $p_{j_{n}}$ (in the enumeration of $\mathfrak{p}$), such that one of the entries of $p_{j_{n}}$ is $ \gamma_{\ell_{n}}$ (or $\beta_{\ell_{n}}$) and:
$$d_{X}(\sigma^{n}(x),\sigma^{n}(y)) = \dfrac{1}{2^{j_{n}}}.$$
By our assumption (that there is no constant subsequence of $x$ neither of $y$) the set $\{ p_{j_{n}}: n \in \mathbb{N}\}$ is infinite and hence $\dfrac{1}{2^{j_{n}}} \rightarrow 0$, what contradicts the inequality $\displaystyle{\limsup_{n \rightarrow \infty} d_{X}(\sigma^{n}(x), \sigma^{n}(y)) >0}$. So, there is at least an edge $\gamma_{m}$ (or $\beta_{m}$) and infinitely many indices $\{i_{k}\}_{k \in \mathbb{N}}$
such that $\gamma_{i_{k}} = \gamma_{m}$ (or $\beta_{i_{k}} = \beta_{m}$) and $\gamma_{i_{k}} \neq \beta_{i_{k}}$ for all $k \in \mathbb{N}$, as desired.

\end{proof}

In the next proposition we continue to characterize the asymptotic behavior of pairs consisting of an infinite sequence and an infinite emitter.

\begin{prop}\label{teoinfemm}
Let $x = \gamma_{1}\gamma_{2}\ldots \in \mathfrak{p}^{\infty}$ be an infinite path and $ A \in \mathfrak{p}^{0}$ a minimal infinite emitter.
Then, 
$$\displaystyle{\liminf_{n \rightarrow \infty} d_{X}(\sigma^{n}(x), A) =0,}$$
if, and only if, there exists infinitely many indices $\{j_{1}, j_{2}, \ldots\}$ such that $\gamma_{j_{k}}\neq \gamma_{j_{\ell}}$ for all $k \neq \ell$, and $s(\gamma_{j_{k}}) \in A$ for all $k \in \mathbb{N}$.
\end{prop}

\begin{proof}

To prove the converse of the proposition, suppose that there are infinitely many pair-wise different edges $\{\gamma_{j_{1}},\gamma_{j_{2}},\ldots\}$ such that $s(\gamma_{j_{i}}) \in A$ for all $i$. Then, for all natural $k$, the first entry of $\sigma^{j_{k}-1}(x)$ is $\gamma_{j_{k}}$. Since $s(\gamma_{j_{k}}) \in A$ for all $k$, $A$ is an initial segment of $\sigma^{j_{k}-1}(x)$, but no initial segment of $\sigma^{j_{k}-1}(x)$ is an initial segment of $A$. Therefore there is an element $p_{n_{k}}$ in the enumeration of $\mathfrak{p}$ such that its first entry is $\gamma_{j_{k}}$ and 
$$d_{X}\left(\sigma^{j_k-1}(x), A\right) = \frac{1}{2^{n_{k}}}.$$
Now, as $\{\gamma_{j_{k}}: k \in \mathbb{N}\}$ is an infinite set formed by pair-wise distinct elements, we have that $\{p_{n_{k}}: k \in \mathbb{N}\}$ is an infinite set formed by pair-wise distinct elements too. Hence $n_{k} \rightarrow \infty$ when $k \rightarrow \infty$, and so $\displaystyle{\liminf_{n \rightarrow \infty} d_{X}(\sigma^{n}(x), A) =0}$ as we wanted. 

For the forward implication of the proposition, first observe that if $x'$ is an infinite path such that $s(x') \notin A$, and $A$ is the element $p_{j}$ in the enumeration of $\mathfrak{p}$, then $A$ is not an initial segment of $x'$ and hence $d_{X}(x', A) \geq \frac{1}{2^{j}}.$ Furthermore, if $x^{n} = \gamma^{n}_{1}\gamma^{n}_{2} \ldots$ is a sequence of infinite paths such that
$\gamma^{n}_{1} = \gamma^{1}_{1}$ for all natural $n$, and $\gamma^{1}_{1}$ is the element $p_{j'}$ in the enumeration of $\mathfrak{p}$, then $d_{X}(x^{n}, A) \geq \min\left\{\frac{1}{2^{j}};\frac{1}{2^{j'}}\right\}$, for all natural $n$.

Finally, suppose that $\displaystyle{\liminf_{n \rightarrow \infty}d_{X}(\sigma^{n}(x), A) = 0}$ and let $\sigma^{n_{k}}(x)$ be a subsequence such that $d_{X}(\sigma^{n_{k}}(x),y) \rightarrow 0$. Denote by $\gamma^{k}_{1}$ the first edge of $\sigma^{n_{k}}(x)$, for all natural $k$. By the observations in the previous paragraph, we have that the sequence $\{\gamma^{k}_{1}\}_{k \in \mathbb{N}}$ has no constant subsequence and $s(\gamma^{k}_{1}) \in A$ for all, except possibly for finite many, $\gamma^{k}_{1}$. Therefore there exists a natural $K$ such that $s(\gamma^{k_{1}}_{1}) \in A$ and  $\gamma^{k_{1}}_{1}\neq \gamma^{k_{2}}_{1}$, for all distinct $k_{1},k_{2} \geq K$.   

\end{proof}

Next we prove the last proposition related to the asymptotic behavior of pairs of infinite paths.

\begin{prop}\label{Teoinf0}
Let $x = \gamma_{1}\gamma_{2}\gamma_{3}\ldots$ and $y = \beta_{1}\beta_{2}\beta_{3}\ldots$ be infinite paths. Then
$$\displaystyle{\liminf_{n \rightarrow \infty} d_{X}(\sigma^{n}(x),\sigma^{n}(y)) = 0}$$
if, and only if, one of the three conditions below is verified:

\begin{enumerate}
\item $\#\{e \in \G^1: \exists i \in \mathbb{N} \text{ such that }  \gamma_{i} = e = \beta_{i}\} = \infty$;
\item There is a subsequence $\{n_{k}\}_{k \in \mathbb{N}}$ with the following property: for all natural $N$, there is $K$ such that $k \geq K$ implies $\sigma^{n_{k}}(x)$ and $\sigma^{n_{k}}(y)$ agree in the first $N$ entries;
\item There are infinitely many indices $\{n_{k}\}_{k \in \mathbb{N}}$ such that $\gamma_{n_{k}} \neq \beta_{n_{k}}$ and $\gamma_{n_{k_{1}}} \neq \gamma_{n_{k_{2}}}$, $\beta_{n_{k_{1}}} \neq \beta_{n_{k_{2}}}$ for all $k_{1} \neq k_{2}$.
\end{enumerate}
\end{prop}
\begin{proof}

We prove the converse first. Suppose that Condition~1 is satisfied. Then, there are subsequences
$\{\sigma^{n_{k}}(x)\}_{k \in \mathbb{N}}$ and $\{\sigma^{n_{k}}(y)\}_{k \in \mathbb{N}}$ such that $\sigma^{n_{k}}(x)_{1} = \sigma^{n_{k}}(y)_{1} = \gamma^{n_{k}}_{1}$ and $\gamma^{n_{k}}_{1} \neq \gamma^{n_{\ell}}_{1}$ for all $k \neq \ell$.  Let $p_{j_{k}}$ in the enumeration of $\mathfrak{p}$ be such that $d_{X}(\sigma^{n_{k}}(x),\sigma^{n_{k}}(y)) = \dfrac{1}{2^{p_{j_{k}}}}$. Notice that the first entry of each $p_{j_{k}}$ is $\gamma^{n_{k}}_{1}$. Since the $\{\gamma^{n_{k}}_{1}: k \in \mathbb{N}\}$ is an infinite set formed by pair-wise distinct elements, the set $\{p_{j_{k}}: k \in \mathbb{N}\}$ is an infinite set formed by pair-wise distinct elements too. So, we conclude that $\displaystyle{\liminf_{n \rightarrow \infty}d_{X}(\sigma^{n}(x),\sigma^{n}(y)) = 0}$.

Next, we assume that Condition~2 holds. Fix $N \in \mathbb{N}$. Since the enumeration of $\mathfrak{p}$ is fixed, there is a natural $N'$ such that if $p_{j}$ is an element in $\mathfrak{p}$ with length greater than $N'$, then $j > N$. By Condition~2 (for $N'$), there is a natural $K$ such that for all $k \geq K$, $\sigma^{n_{k}}(x)$ and $\sigma^{n_{k}}(y)$ have the same $N'$ entries. Hence, for all $k\geq K$, any initial segment (say $p_j$) of one of $\sigma^{n_{k}}(x)$ or $\sigma^{n_{k}}(y)$, but not the other, must be such that $j > N$. Therefore, for $k \geq K$, we have that
$$d_{X}(\sigma^{n_{k}}(x),\sigma^{n_{k}}(y)) < \dfrac{1}{2^{N}}.$$
As $N$ was arbitrary, we finished this part.

Suppose that the Condition 3 holds. Then, there are subsequences $\{\sigma^{n_{k}}(x)\}_{k \in \mathbb{N}}$ and $\{\sigma^{n_{k}}(y)\}_{k \in \mathbb{N}}$ such that 
$\sigma^{n_{k}}(x)_{1} = \gamma^{n_{k}}_{1}$, $\sigma^{n_{k}}(y)_{1} = \beta^{n_{k}}_{1}$,
$\gamma^{n_{k}}_{1} \neq \beta^{n_{k}}_{1}$ and $\gamma^{n_{k_{1}}}_{1} \neq \gamma^{n_{k_{2}}}_{1}$, $\beta^{n_{k_{1}}}_{1} \neq \beta^{n_{k_{2}}}_{1}$ for all $k_{1} \neq k_{2}$ and $k \in \mathbb{N}$. Let $p_{j_{k}} \in \mathfrak{p}$ be such that $d_{X}(\sigma^{n_{k}}(x),\sigma^{n_{k}}(y)) = \dfrac{1}{2^{j_{k}}}.$ Notice that the first edge of $p_{j_{k}}$ is either $\gamma^{n_{k}}_{1}$ or $\beta^{n_{k}}_{1}$. Since $\{\gamma^{n_{k}}_{1}: k \in \mathbb{N}\}$ and $\{\beta^{n_{k}}_{1}: k \in \mathbb{N}\}$ are both infinite sets formed by pair-wise distinct elements, we infer the same for the set $\{p_{j_{k}}: k \in \mathbb{N}\}$. So we conclude that $d_{X}(\sigma^{n_{k}}(x),\sigma^{n_{k}}(y))\rightarrow 0$, and hence $\displaystyle{\liminf_{n \rightarrow \infty}d_{X}(\sigma^{n}(x),\sigma^{n}(y)) = 0}$.

Now we prove the forward implication of the proposition. Suppose that $$\displaystyle{\liminf_{n \rightarrow \infty} d_{X}(\sigma^{n}(x),\sigma^{n}(y)) = 0},$$ and, additionally, suppose that Conditions~1 and 2 are not satisfied. We will show that, under these hypothesis, Condition~3 holds.

Let $\{m_{k}\}_{k \in \mathbb{N}}$ be a subsequence such that $d_{X}(\sigma^{m_{k}}(x),\sigma^{m_{k}}(y)) \rightarrow 0$. As Condition~2 is not verified, passing to a subsequence of $\{m_{k}\}_{k \in \mathbb{N}}$ if necessary, we have that there is $\ell \in \mathbb{N}$, $\ell\geq 0$, such that, for all natural $k$, $\sigma^{m_{k}}(x)$ and $\sigma^{m_{k}}(y)$ are identical in the first $\ell$ entries and are different in the $(\ell+1)$-th entry.

Now,  let $s_{k}$ be the initial segment of length $\ell$ of $\sigma^{m_{k}}(x)$ (and of $\sigma^{m_{k}}(y)$ too, as $\sigma^{m_{k}}(y)$ agrees with $\sigma^{m_{k}}(x)$ in the first $\ell$ entries). As Condition~1 is not satisfied, the set $\{s_{k}: k \in \mathbb{N} \}$ is finite. Denote by $\gamma^{n_{k}}_{\ell + 1}$ and $\beta^{n_{k}}_{\ell + 1}$, respectively, the $(\ell+1)$-th entry of $\sigma^{m_{k}}(x)$ and $\sigma^{m_{k}}(y)$ (so $\gamma^{n_{k}}_{\ell + 1} \neq \beta^{n_{k}}_{\ell + 1}$), and suppose that  $s_{k}\gamma^{n_{k}}_{\ell + 1}$ and $s_{k}\beta^{n_{k}}_{\ell + 1}$ are in the positions $j_{(k,x)}$ and $j_{(k,y)}$ in the enumeration of $\mathfrak{p}$, respectively. Then
$$\max\left\{\frac{1}{2^{j_{(k,x)}}};\frac{1}{2^{j_{(k,y)}}}\right\} \leq d_{X}(\sigma^{m_{k}}(x),\sigma^{m_{k}}(y))\rightarrow 0, $$
and hence we conclude that $\gamma^{n_{k_{1}}}_{\ell + 1} \neq\gamma^{n_{k_{2}}}_{\ell + 1}$ and $\beta^{n_{k_{1}}}_{\ell + 1} \neq  \beta^{n_{k_{2}}}_{\ell + 1}$ for all $k_{1} \neq k_{2}$ except possibly for finitely many pairs $k_{1} \neq k_{2}$. Therefore Condition~3 is verified.
\end{proof}

We are now ready to prove our main results. In Theorem~\ref{teothebest} below we construct an uncountable scrambled set ($S'_{\alpha}$) in the shift space $\{0,1\}^{\mathbb{N}}$ (provided it can be embedded in the ultragraph shift space). Also, we describe the set $S'_\alpha$, by showing that such set is not closed and its closure, $\overline{S'_{\alpha}}$, is not a scrambled set. At the end of the proof, we extract a perfect subset of $S'_{\alpha}$.

Before we proceed we need the following definition regarding ultragraphs.

\begin{defin}\label{m,n}
Let $\mathcal{G}$ be an ultragraph.  A \textit{closed path based at the vertex $v$} is a finite path $e_{1}e_{2}\ldots e_{k}$ such that $v = s(e_{1}) \in r(e_{k})$ and $s(e_{i}) \neq v$ for all $i > 1$. We denote by $CP_{\mathcal{G}}(v)$ the set of all closed paths in $\mathcal{G}$ based at $v$.

\end{defin}
 
 \begin{obs} Notice that in the language of Leavitt path algebras what we are calling closed path based at vertex $v$ is known as closed simple path based at vertex $v$, see \cite{AAS}.
 \end{obs}

\begin{teo}\label{teothebest}
Let $\G$ be an ultragraph and suppose that there exists a vertex $v$ such that $\#CP_{\mathcal{G}}(v) \geq 2$. Then the associated shift space $X$ is Li-Yorke chaotic. Furthermore, $X$ contains an uncountable scrambled set which is perfect and compact, as well as an uncountable scrambled set that is not closed and whose closure is not a scrambled set.
\end{teo}

\begin{proof}
Let $\mathcal{G}$ be an ultragraph such that $\#CP_{\mathcal{G}}(v) \geq 2$ for, at least, a vertex $v$. By Definition~\ref{m,n}, we can find two different closed paths $c_{1}$ and $c_{2}$ based at $v$. Notice that the set of all infinite paths formed by $c_{1}$ and $c_{2}$, $\{c_{1},c_{2}\}^{\mathbb{N}}$, is contained, and compact, in the ultragraph shift space $X$. We denote $c_{1}$ by $0$ and $c_{2}$ by $1$, and work with $\{0,1\}^{\mathbb{N}}$ instead of $\{c_{1},c_{2}\}^{\mathbb{N}}$, in order to make the notation lighter.

Now, for each natural $n$, let 
\begin{equation}\label{eqI}
\displaystyle{a_{n} = \sum_{i=1}^{n}(i+1)}
\end{equation}
and define $I \subset \mathbb{N}$ by
$\displaystyle{I := \left\{a_{n}: n \in \mathbb{N} \right\}.}$ 
Observe that $a_{1} = 2$ and $a_{n} = a_{n-1} + n + 1$ for $n \geq 2$. Furthermore, $1 \notin I$ and, for each $n \geq 2$, the set $I$ contains the elements $a_{n-1}$ and $a_{n}$ but does not contain the following set of consecutive natural numbers: $\{a_{n-1} +1, a_{n-1} + 2, \ldots, a_{n-1} + n\}.$ Therefore:
$$\mathbb{N}-I = \{\underbrace{1,}_{\mbox{one entry}}  \underbrace{3, 4,}_{\mbox{two}} \underbrace{6,7,8,}_{\mbox{three}} \underbrace{10, 11, 12, 13,}_{\mbox{four}} \underbrace{15, 16, 17, 18, 19,}_{\mbox{five}} \underbrace{21,22,23,24,25,26,}_{\mbox{six}}28\ldots\}.$$

Next, fix an element $\alpha \in \{0,1\}^{\mathbb{N}}$ and define
the set $S_{\alpha} \subset \{0,1\}^{\mathbb{N}}$ by:
\begin{equation}\label{defS}
S_{\alpha} := \left\{\beta \in \{0,1\}^{\mathbb{N}}: \beta_{i} = \alpha_{i}, \forall i \in (\mathbb{N}-I) \right\}.
\end{equation}

Note that, if $\beta^{1}, \beta^{2} \in S_{\alpha}$ and $a_{k} \in I$ (see (\ref{eqI}) above) then $\sigma^{a_{k}}(\beta^{1})$ and $\sigma^{a_{k}}(\beta^{2})$ agree in the first $k+1$ entries, for all natural $k$. So, from the second condition in Proposition~\ref{Teoinf0}, we have that
$$\displaystyle{\liminf_{n \rightarrow \infty} d_{X}(\sigma^{n}(\beta^{1}), \sigma^{n}(\beta^{2})) = 0.}$$

Observe that $S_{\alpha}$ is not a scrambled set. For instance if we choose $\beta^{1} = \alpha_1 \alpha_2 \ldots \in S_\alpha$, and let $\beta^{2}$ be any infinite path in $S_{\alpha}$ such that $\beta^{2}_{i} = \alpha_{i}$, except for a non-zero finite number of $i$, then  
$\lim d_{X}(\sigma^{n}(\beta^{1}), \sigma^{n}(\beta^{2})) = 0$.
Hence $\displaystyle{\limsup_{n \rightarrow \infty} d_{X}(\sigma^{n}(\beta^{1}), \sigma^{n}(\beta^{2})) = 0}$ and $(\beta^{1}, \beta^{2})$ is not a scrambled pair. Therefore we need to ``select" elements in $S_{\alpha}$ in order to construct a uncountable scrambled set. More precisely, we will extract a subset $S'_{\alpha} \subset S_{\alpha}$ that satisfies the following property:
$$\displaystyle{\beta^{1}, \beta^{2} \in S'_{\alpha}, \beta^{1} \neq \beta^{2} \Rightarrow \limsup_{n \rightarrow \infty} d_{X}(\sigma^{n}(\beta^{1}),\sigma^{n}(\beta^{2}))>0}.$$
Clearly such $S'_{\alpha}$ is a scrambled set. Next we show how to obtain such $S'_\alpha$.

By the definition of $S_{\alpha}$,  any element $\beta \in S_{\alpha}$ is determined by its entries indexed by $I$, because the other entries are already fixed. Therefore 
we have a bijection from $\{0,1\}^{\mathbb{N}}$ to  $S_{\alpha}$ defined by $\gamma \mapsto \beta^{\gamma},$
where the path $\beta^{\gamma}$ is given by: 
\begin{equation}\label{eqgbeta}
\beta_{i}^{\gamma} = \begin{cases}
           \alpha_{i}  & \text{if } i \in (\mathbb{N} - I),\\
             \gamma_{k}  & \text{if } i = a_{k} \in I.
       \end{cases} \quad
\end{equation}
Denote by $\mathcal{P}(\mathbb{N^{*}})^{\infty}$ the set of all infinite subsets of $\mathbb{N}^{*}$. Enumerate each $J \in \mathcal{P}(\mathbb{N^{*}})^{\infty}$ in an increasing order, that is, write $J = \{j_{1},j_{2}, \ldots : j_i< j_{i+1} \ \forall i \}$. Consider the following sequence associated to $J$: $ (j_{1}, j_{1}, j_{2}, j_{1},j_{2},j_{3}, j_{1},j_{2},j_{3},j_{4}, j_{1}, \ldots)$. Observe that different elements of $\mathcal{P}(\mathbb{N^{*}})^{\infty}$ induce different sequences. Define the function $f: \mathcal{P}(\mathbb{N^{*}})^{\infty} \rightarrow \{0,1\}^{\mathbb{N}}$  by:
\begin{equation}\label{eqf}
f(J) = (\underbrace{0,\dots,0,}_{j_{1}}\underbrace{1,\ldots,1,}_{j_{1}}\underbrace{0,\dots,0,}_{j_{2}}\underbrace{1,\ldots,1,}_{j_{1}}\underbrace{0,\dots,0,}_{j_{2}}\underbrace{1,\ldots,1,}_{j_{3}}\underbrace{0,\dots,0,}_{j_{1}}\underbrace{1,\ldots,1,}_{j_{2}}\ldots).
\end{equation}

Next we prove a result stronger than the injectivity of $f$: if $J_{1}$ and $J_{2}$ are distinct infinite sets belonging to $\mathcal{P}(\mathbb{N}^{*})$, then the infinite paths $f(J_{1})$ and $f(J_{2})$ are different in infinitely many entries. 
To begin, choose $J$ such that $1 \in J$. Then $j_{1} = 1$ and, according to our definition of $f$, the segments ``010'' and ``101'' appear infinitely many times in the path $f(J)$. Even more: only sets that contain the element 1 satisfy this property, that is, if $J'$ is an element of $\mathcal{P}(\mathbb{N^{*}})^{\infty}$ such that $1 \notin J'$, then ``010'' and ``101'' are not segments in the infinite path $f(J')$. More generally, given $j \in \mathbb{N}^{*}$, only sets $J$ which contain $\{j\}$ satisfy that ``$0\underbrace{11\ldots1}_{j}$0'' and ``$1\underbrace{00\ldots0}_{j}1$'' are segments in $f(J)$, and these segments appear infinitely many times in $f(J)$. So, given $J_{1} \neq J_{2}$, there is a natural $j$ that belongs to only one of the sets, say $j \in J_{1}\cap J_2^c$. Then, according to the above, the segments `$0\underbrace{11\ldots1}_{j}$0'' and ``$1\underbrace{00\ldots0}_{j}1$'' appear infinitely many times in the path $f(J_{1})$ and never appear in $f(J_{2})$. This guarantees that $f(J_{1})$
and $f(J_{2})$ differ in infinitely many entries, as we wanted.

Before we proceed we need to prove the following claim:

{\bf Claim 1:} 
Let $J_{1},J_{2}$ be two sets belonging to $\mathcal{P}(\mathbb{N}^{*})^{\infty}$ and write $J_{k} = \{j^{k}_{1},j^{k}_{2},\ldots: j_{1}^{k}<j_{2}^{k}<j_{3}^{k}< \ldots\}$, $k \in \{1,2\}$. Define $\displaystyle{i_{k,n} := \sum_{i=1}^{n}(n-i+1)j^{k}_{i}}$, for $k = 1,2$. Then, for all natural $n$, $f(J_{1})$ and $f(J_{2})$  agree in the first $i_{1,n}$ entries if, and only if, $j_{i}^{1} = j_{i}^{2}$ for all $i \leq n$. Summarizing: 
\begin{equation}\label{propriedade}
f(J_{1})_{i} = f(J_{2})_{i}, \ \forall i \leq i_{1,n} \Longleftrightarrow j_{i}^{1} = j_{i}^{2}, \  \forall i \leq n
\end{equation}
We prove the converse first. Let $J_{1}$ and $J_{2}$ be sets in $\mathcal{P}(\mathbb{N}^{*})^{\infty}$ and write then as in the hypothesis of the Lemma. If $j_{i}^{1} = j_{i}^{2}$ for all $i \leq n$ then, by equality (\ref{eqf}), we have that $f(J_{1})$ and $f(J_{2})$
agree in the first $j_{1}^{1}+(j_{1}^{1}+j^{1}_{2})+\ldots+(j_{1}^{1}+j^{1}_{2}+\ldots+j^{1}_{n})$ entries. But:
$$\displaystyle{j_{1}^{1}+(j_{1}^{1}+j^{1}_{2})+\ldots+(j_{1}^{1}+j^{1}_{2}+\ldots+j^{1}_{n}) = \sum_{i=1}^{n}(n-i+1)j^{k}_{i} =  i_{k,n}}.$$
Now we prove the forward implication of the property (\ref{propriedade}). Suppose that $f(J_{1})_{i} = f(J_{2})_{i}$ for all $i \leq i_{1,n}$. If $J_{1} = J_{2}$ there is nothing to do. So, we assume that $J_{1}$ and $J_{2}$ are distinct sets. Suppose (by contradiction) that $n_{0}:=\min\{i:j_{i}^{1} \neq j_{i}^{2}\} < n$ and assume that $j_{n_{0}}^{1}<j_{n_{0}}^{2}$. Then, by equality (\ref{eqf}) again, we have that $f(J_{1})$ and $f(J_{2})$ agree in the first $i_{1,n_{0}}$ entries, but they disagree in the $(i_{1,n_{0}}+1)$-th entry. As $n_{0} < n$, we must have $(i_{1,n_{0}}+1) \leq i_{1,n}$, what contradicts our assumption ($f(J_{1})_{i} = f(J_{2})_{i}$ for all $i \leq i_{1,n}$). Hence $n_{0} \geq n$ as we wanted and the claim is proved.

Now we get back to mainline of the proof of the theorem. Consider the following set (see (\ref{eqgbeta})): 
\begin{equation}\label{defS'}
S'_{\alpha}:= \{(\beta^{f(J)}: J \in \mathcal{P}(\mathbb{N^{*}})^{\infty} \}.
\end{equation}
By the paragraph above Claim~1, if $J_{1} \neq J_{2}$ then the paths 
$\beta^{f(J_{1})}$ and $\beta^{f(J_{2})}$ differ in infinitely many entries. Furthermore, since $\beta^{f(J_{1})}$ and $\beta^{f(J_{2})}$ can be seen as infinite paths in $\{0,1\}^{\mathbb{N}}$, we obtain, from Proposition~\ref{Teolimsup0}, that: 
$$\displaystyle{\limsup_{n \rightarrow \infty} d_{X}(\sigma^{n}(\beta^{f(J_{1})}),\sigma^{n}(\beta^{f(J_{2})}))>0},$$
for all $\beta^{f(J_{1})} \neq \beta^{f(J_{2})}$. As we already know that 
$\displaystyle{\liminf_{n \rightarrow \infty}d_{X}(\sigma^{n}(\beta^{1}),\sigma^{n}(\beta^{2}))=0}$ for all $\beta^{1}, \beta^{2} \in S_{\alpha}$, we also have that $\displaystyle{\liminf_{n \rightarrow \infty} d_{X}(\sigma^{n}(\beta^{f(J_{1})}),\sigma^{n}(\beta^{f(J_{2})}))=0}$ for all $\beta^{f(J_{1})},\beta^{f(J_{2})} \in S'_{\alpha}$. Hence $S'_{\alpha}$ is scrambled. Furthermore, $S'_{\alpha}$ is uncountable, since it is the image of the function $J \mapsto \beta^{f(J)}$, and both $f$ and $\gamma \mapsto \beta^{\gamma}$ are injective functions from $\mathcal{P}(\mathbb{N}^{*})^{\infty}$. So $S'_{\alpha}$ is an uncountable scrambled set and $X$ is Li-Yorke chaotic.

Next we show that the set $S'_{\alpha}$ is not closed, has no isolated points and all its adherent points are infinite paths. Furthermore, we show that $\overline{S'_{\alpha}}$ is not a scrambled set in $X$.

Consider the infinite paths $\beta^{1}, \beta^{2} \in \{0,1\}^{\mathbb{N}}$
given by: 
\begin{equation*}
\beta_{i}^{1} = \begin{cases}
           \alpha_{i}  & \text{if }  i \in (\mathbb{N} - I),\\
             0  & \text{if } i = a_{k} \in I
       \end{cases} \quad
\end{equation*}
and
\begin{equation*}
\beta_{i}^{2} = \begin{cases}
           \alpha_{i}  & \text{if } i \in (\mathbb{N} - I),\\
             1 & \text{if } i =5,\\
             0 & \text{if } i \in I - \{5\}.
       \end{cases} \quad
\end{equation*}
Note that any sequence of sets $J_{n} = \{j_{1}^{n},j_{2}^{n},j_{3}^{n}\ldots\}$ such that $j_{1}^{n} < j_{2}^{n} < j_{3}^{n}< \ldots$ and  $j_{1}^{n}\rightarrow \infty$ satisfies $\beta^{f(J_{n})} \rightarrow \beta^{1}$, and any sequence of sets $H_{n} = \{h_{1}^{n},h_{2}^{n},h_{3}^{n},\ldots\}$ such that $h_{1}^{n}< h_{2}^{n}< h_{3}^{n},\ldots$, $h_{1}^{n} = 1$ and $h_{2}^{n} \rightarrow \infty$ satisfies $\beta^{f(H_{n})} \rightarrow \beta^{2}$. 
So $\beta^{1},\beta^{2} \in \overline{S'_{\alpha}}$, but $\beta^{1},\beta^{2} \notin S'_{\alpha}$. Hence $S'_{\alpha}$ is not closed. Even more, as 
$\sigma^{5}(\beta^{1}) = \sigma^{5}(\beta^{2})$, it follows that $\displaystyle{\lim_{n \rightarrow \infty} d_{X}(\sigma^{n}(\beta^{1}),\sigma^{n}(\beta^{2})) = 0}$. So $(\beta^{1},\beta^{2})$ is not a scrambled pair and $\overline{S'_{\alpha}}$ is not a scrambled set.
Recall that $S_{\alpha}$, defined in (\ref{defS}), is a closed subset of $X$ formed only by infinite paths. As $S'_{\alpha} \subset S_{\alpha}$, all adherent points of $S'_{\alpha}$ are infinite paths.
Finally, to prove that $S'_{\alpha}$ has no isolated points, fix $\beta^{f(J)} \in S'_{\alpha}$ such
that $J \in \mathcal{P}(\mathbb{N^{*}})^{\infty}$ and write $J = \{j_{1},j_{2}, \ldots : j_i< j_{i+1} \ \forall i \}$. Define $j'_{n} := j_{n} + 1$ for all $n \in \mathbb{N}$, and consider the sequence $\{\beta^{f(J_{n})}\}$, where 
\begin{equation}\label{eqj'}
J_{n} = \{j_{1}, j_{2},j_{3}, \ldots, j_{n},j'_{n+1},j'_{n+2},j'_{n+3},\ldots\} 
\end{equation}
for all 
$n \in \mathbb{N}$. Then $\beta^{f(J_{n})} \rightarrow \beta^{f(J)}$, but $\beta^{f(J_{n})} \neq \beta^{f(J)}$ for all natural $n$. So $S'_{\alpha}$ has no isolated points.

Although $S'_{\alpha}$ has no isolated points, it is not closed, so it is not perfect neither compact. So, in the paragraphs below, we extract a subset $S''_{\alpha} \subset S'_{\alpha}$ which is still uncountable and scrambled, and is also perfect (and compact).

As before, we write each $J \in \mathcal{P}(\mathbb{N}^{*})^{\infty}$ as $J = \{j_{1}, j_{2}, \ldots: j_i< j_{i+1} \ \forall i\}$. Define the following set:
$$\mathcal{P}:= \{J \in \mathcal{P}(\mathbb{N}^{*})^{\infty}: j_{n} \in \{2n - 1; 2n\}, \forall n \in \mathbb{N}\}.$$
Note that $\mathcal{P}$ is an uncountable subset of $\mathcal{P}(\mathbb{N}^{*})^{\infty}$. According to equalities (\ref{eqgbeta}) and (\ref{eqf}), define the following set:
$$S''_{\alpha} := \{\beta^{f(J)}: J \in \mathcal{P}\}.$$

Clearly $S''_{\alpha}$ is a subset of $S'_{\alpha}$ (see equation (\ref{defS'})). We show that $S''_{\alpha}$ is closed: Let $\beta$ be a path that is the limit of a sequence $\{\beta^{f(J_{m})}\}_{m \in \mathbb{N}}$ of paths from $S''_{\alpha}$, where $J_{m} = \{j_{1}^{m},j_{2}^{m},j_{3}^{m},\ldots : j_i^m< j_{i+1}^m \ \forall i\}$, for all natural $m$. As $S''_{\alpha}$ is a subset of $\{0,1\}^{\mathbb{N}}$ (which is closed in the shift space $X$), we have  $|\beta| = \infty$. Let $n \in \mathbb{N}$. Then there is a natural $N_{n}$ such that $m \geq N_{n}$ implies $\beta_{i}^{f(J_{m})} = \beta_{i}$ for all $i \in \{1,2,3, \ldots, a_{k_{n}}\}$, where $\displaystyle{a_{m} = \sum_{i=1}^{m}(i+1)}$ and  $\displaystyle{k_{n} = 2\sum_{i=1}^{n}(n +1 -i)i}$. In particular, $f(J_{m})_{i}=\beta^{f(J_{m})}_{a_{i}} = \beta_{a_{i}}$ for all natural $m$ and all $i \leq k_{n}$. Then, $f(J_{m_{1}})_{i} = f(J_{m_{2}})_{i}$ for all $m_{1},m_{2} \geq N_{n}$ and $i \leq k_{n}$.  As $j_{i}^{m}\leq 2i$ for all naturals $i$ and $m$ we have, for all $m$:
$$\displaystyle{\sum_{i=1}^{n}(n-i+1)j_{i}^{m} \leq \sum_{i=1}^{n}(n-i+1)2i = k_{n}}.$$
Then, $f(J_{m_{1}})_{i} = f(J_{m_{2}})_{i}$ for all $\displaystyle{i \leq i_{m_{1},n} := \sum_{i=1}^{n}(n-i+1)j_{i}^{m_{1}}}$. Hence, by Claim~1 (see \ref{propriedade}), we conclude that $j^{m_{1}}_{i} = j^{m_{2}}_{i} = j_{i}^{N_{i}}$, for all $m_{1},m_{2} \geq N_{n}$ and $i \in \{1,2,3,\ldots, n\}$. Let $J := \{j^{N_{1}}_{1},j^{N_{2}}_{2}, j^{N_{3}}_{3}, \ldots\}$. We show that $\beta = \beta^{f(J)}$.  In fact, for $a_{m} = \sum_{i=1}^{m}(i+1)$ we have:
$$\beta_{a_{m}}^{f(J)} = f(J)_{m} = f(J_{N_{m}})_{m} = \beta_{a_{m}}^{f(J_{N_{m}})} = \beta_{a_{m}}.$$
Finally, for all natural $m' \notin I$, we have $\beta^{f(J_{m})}_{m'} = \alpha_{m'} = \beta_{m'}^{f(J)}$ for all $m \in \mathbb{N}$. As $\displaystyle{\lim_{m \rightarrow \infty}\beta^{f(J_{m})} = \beta}$, we have $\beta_{m'} = \alpha_{m'}$. Also, as we already know that $j_{i}^{N_{i}} \in \{2i-1,2i\}$ for all natural $i$, we have that $\beta \in S''_{\alpha}$ and $S''_{\alpha}$ is closed.

To verify that $S''_{\alpha}$ has no isolated points we use the same idea as with $S'_{\alpha}$ (see equality (\ref{eqj'})).  Let $\beta^{f(J)} \in S''_{\alpha}$ be a path where $J = \{j_{1},j_{2},j_{3}, \ldots: j_i <j_{i+1} \ \forall i\}$. For each natural $n$, if $j_{n} = 2n-1$ choose $j'_{n} = 2n$, and if $j_{n} = 2n$ choose $j'_{n} = 2n-1$, and consider the sequence $\beta^{f(J_{n})}$ such that $J_{n} = \{j_{1}, j_{2},j_{3}, \ldots, j_{n},j'_{n+1},j'_{n+2},j'_{n+3},\ldots\}$ for all 
$n \in \mathbb{N}$. Then, $\beta^{f(J_{n})} \rightarrow \beta^{f(J)}$, but $\beta^{f(J_{n})} \neq \beta^{f(J)}$ for all natural $n$. So, $S''_{\alpha}$ has no isolated points, as we wanted. 

Finally, notice that $S''_{\alpha}$ is compact, since it is a perfect subset of the compact subset $\{0,1\}^{\mathbb{N}} \cong \{c_1, c_2\}^{\mathbb{N}} \subseteq X$.

\end{proof}

As we mentioned in the introduction, Raines and Underwood have shown in \cite{Raines} that the existence of a scrambled pair implies Li-Yorke chaos in the (product topology) shift space $X$ generated by an infinite graph $\mathcal{G}$, but such phenomena does not occur in general for (product topology) labelled edge shifts. In our context, we show that the existence of a scrambled pair formed by infinite paths imply Li-Yorke chaos (Proposition~\ref{teopair}). Using this we obtain a result analogous to the one in \cite{Raines}: When the ultragraph $\G$ is just a graph, the existence of a scrambled pair implies Li-Yorke Chaos (Proposition~\ref{teopair2}). We also give an example of an ultagraph shift space (not associated to a graph) where the existence of a scrambled pair does not imply Li-Yorke chaos (Example \ref{exedoesnot}). Before we proceed we need the following lemma.

\begin{lema}\label{lemcpse} Let $\mathcal{G}$ be an ultragraph, $X$ be the associated shift space, and let $x = \gamma_{1}\gamma_{2}\gamma_{3}\ldots \in X$. If there exist an edge $e_{1}$ and a sequence of the naturals $\{i_{k}\}_{k \in \mathbb{N}}$ such that $\gamma_{i_{k}}=e_{1}$ for all natural $k$, then $CP_{\mathcal{G}}(s(e_{1})) \neq \emptyset$. Additionally, if in this case $CP_{\mathcal{G}}(s(e_{1}))$ is a singleton, say $CP_{\mathcal{G}}(s(e_{1})) = \{e_{1}e_{2}e_{3} \ldots e_{m}\}$, then there exists a finite path $\gamma$ such that $$x = \gamma e_{1}e_{2}\ldots e_{m}e_{1}e_{2}\ldots e_{m}e_{1}e_{2}\ldots e_{m}\ldots.$$
\end{lema}
\begin{proof}
Let $x = \gamma_{1}\gamma_{2}\gamma_{3}\ldots \in X$ and suppose that 
$\gamma_{i_{k}}=e_{1}$ for all natural $k$.  Let $J_{s(e_{1})} = \{j \in \mathbb{N}: \gamma_{j} = e_{1}\}$ and enumerate it in the increasing order, that is, write  $J_{s(e_{1})}= \{j_{1},j_{2},\ldots : j_i<j_{i+1}\}$. Fix a natural $\ell > 0$ and consider the path $c_{\ell} = \gamma_{j_{\ell}}\gamma_{j_{\ell}+1}\ldots \gamma_{j_{(\ell+1)}-1}$. Note that $\gamma_{j_{\ell}} = e_{1} = \gamma_{j_{(\ell+1)}}$ and $s(e_{1}) \in r(\gamma_{j_{(\ell+1)}-1})$. Also, if $j_{\ell} < j < j_{(\ell+1)}$, then $\gamma_{j} \neq e_{1}$.  

If $c_{\ell}$ is itself a closed path based at $s(e_{1})$, then $c_{\ell} \in CP_{\mathcal{G}}(s(e_{1}))$ and we are done. On the other hand, if $c_{\ell}$ is not a closed path based at $s(e_{1})$ then, by Definition \ref{m,n}, the set $$J'_{s(e_{1})} :=\{ j \in \mathbb{N}:j_{\ell} < j <  j_{(\ell+1)} \mbox{ and } s(\gamma_{j}) = s(e_{1})\}$$ must be non empty. Let  $j_{\min} = \min \{j \in J'_{s(e_{1})}\}$. Then the path $c'_{\ell}:= \gamma_{j_{\ell}}\ldots\gamma_{(j_{\min}-1)}$ is a closed path based at $s(e_{1})$, and hence $c'_{\ell} \in CP_{\mathcal{G}}(s(e_{1}))$ and we finished the first part.

For the secont part, assume additionally that $CP_{\mathcal{G}}(s(e_{1})) = \{e_{1}e_{2}e_{3} \ldots e_{m}\}$. We will prove that the path $c_{\ell} =\gamma_{j_{\ell}}\gamma_{j_{\ell}+1}\ldots \gamma_{(j_{(\ell+1)}-1)}$ must be a closed path based at $s(e_{1})$. Suppose the opposite, that is, the path $c_{\ell}$ is not a closed path based at $s(e_{1})$. So, as we have done before, $c'_{\ell}= \gamma_{j_{\ell}}\ldots\gamma_{(j_{\min}-1)}$ is a closed path based at $s(e_{1})$. Then, $c'_{\ell} = e_{1}e_{2}\ldots e_{m}$. In the same way, there will be a closed path based at $s(e_{1})$, namely $c''_{\ell} = \gamma_{j_{\min}}\ldots\gamma_{j'}$, for some $j_{\min} \leq j'\leq j_{\ell}-1$. So, again $c''_{\ell} = e_{1}e_{2}\ldots e_{m}$ and then $\gamma_{j_{\min}} = e_{1}$. As $j_{\ell} < j_{\min} < j_{(\ell+1)}$, we have a contradiction with the increasing enumeration of $J_{s(e_{1})}$ we have written (that is, there is no element of $J_{s(e_{1})}$ greater than $j_{\ell}$ and less than $j_{(\ell+1)}$). So, $c_{\ell} \in CP_{\mathcal{G}}(s(e_{1}))$ and then $c_{\ell} = e_{1}e_{2}\ldots e_{m}$. As $\ell \in \mathbb{N}$ was arbitrary, if we set $\gamma:=\gamma_{1}\gamma_{2}\gamma_{3}\ldots\gamma_{(j_{1}-1)}$, we have:
$$x = \gamma e_{1}e_{2}\ldots e_{m}e_{1}e_{2}\ldots e_{m}e_{1}e_{2}\ldots e_{m}\ldots,$$
and this finishes the proof.
\end{proof}

We can now prove the following.

\begin{teo}\label{teopair}
An ultragraph shift space presents Li-Yorke chaos if, and only if, it has a scrambled pair formed by infinite paths.
\end{teo}
\begin{proof}
Let $\G$ be an ultragraph, $X$ be the associated shift space, and let $x,y \in X$ be a scrambled pair such that $x = \gamma_{1}\gamma_{2}\gamma_{3}\ldots$ and $y = \beta_{1}\beta_{2}\beta_{3}\ldots$. We will show that there is a vertex $v$ in $\mathcal{G}$ such that $\#CP_{\mathcal{G}}(v) \geq 2$.

By the definition of scrambled pair we have that $\displaystyle{\limsup_{n \rightarrow \infty} d_{X}(\sigma^{n}(x),\sigma^{n}(y)) > 0}$. So, by Proposition~ \ref{Teolimsup0}, we get a subsequence $\{i_{k}\}_{k \in \mathbb{N}}$ such that $\gamma_{i_{k}} = e_{1}$, and $\gamma_{i_{k}} \neq \beta_{i_{k}}$ for all $k \in \mathbb{N}$. Then, by Lemma \ref{lemcpse}, there is a closed path based at $s(e_{1})$, say $e_{1}e_{2}\ldots e_{m}$. So, $\#CP_{\mathcal{G}}(s(e_{1})) \geq 1$.

Assume, by contradiction, that $CP_{\mathcal{G}}(s(e_{1}))$ is a singleton, that is, $CP_{\mathcal{G}}(s(e_{1})) = \{e_{1}e_{2}\ldots e_{m}\}$. Again, by Lemma \ref{lemcpse}, there is a finite path $\gamma$ such that:
\begin{equation}\label{eqxg}
x = \gamma e_{1}e_{2}\ldots e_{m}e_{1}e_{2}\ldots e_{m}e_{1}e_{2}\ldots e_{m}\ldots.
\end{equation}

Using again that $x,y$ is a scrambled pair, we have that $\displaystyle{\liminf_{n \rightarrow \infty} d_{X}(\sigma^{n}(x),\sigma^{n}(y)) = 0}$. But, as $x$ is formed by finitely many edges, we obtain (see Proposition~\ref{Teoinf0}) a subsequence $\{n_{k}\}_{k \in \mathbb{N}}$ such that $\sigma^{n_{k}}(x)$ and $\sigma^{n_{k}}(y)$ have the same first $k$ entries, for all natural $k$. This implies that there is a sequence of the naturals, say $\{\ell_{k}\}_{k \in \mathbb{N}}$, such that $\beta_{l_{k}} = e_{1}$ for all natural $k$. Therefore, by Lemma~\ref{lemcpse} (and since we are assuming that $CP_{\mathcal{G}}(s(e_{1}))$ is a singleton), there exists a finite path $\gamma'$ such that: 
\begin{equation}\label{eqyg}
y = \gamma' e_{1}e_{2}\ldots e_{m}e_{1}e_{2}\ldots e_{m}e_{1}e_{2}\ldots e_{m}\ldots.
\end{equation}

Now notice that $\sigma^{|\gamma|}(x) = \sigma^{|\gamma'|}(y) = e_{1}e_{2}\ldots e_{m}e_{1}e_{2}\ldots e_{m}e_{1}e_{2}\ldots e_{m}\ldots$. We are then left with two possibilities. Either there is a natural $n$ such that $\sigma^{n}(x) = \sigma^{n}(y)$, and in this case $x,y$ is not a scrambled pair, or $\sigma^{n}(x) \neq \sigma^{n}(y)$ for all natural $n$. In this last case, let $e_{k}$ be the element $p_{m_{k}}$ in the enumeration of $\mathfrak{p}$, (see the definition of the metric (\ref{definmetricanova})). Then, by equations (\ref{eqxg}) and (\ref{eqyg}), for any natural $n \geq \max\{|\gamma|,|\gamma'|\}$, we have
$$d_{X}(\sigma^{n}(x),\sigma^{n}(y))\geq \min \left\{\dfrac{1}{2^{m_{k}}}: 1\leq k \leq m\right\},$$ 
and hence (again) $x,y$ is not a scrambled pair, a contradiction. We conclude that $\# CP_{\mathcal{G}}(s(e_{1}))\geq 2$ and the result now follows from Theorem~\ref{teothebest}.


\end{proof}

\begin{prop}\label{teopair2}
Let $\mathcal{G}$ be a graph. Then the associated shift space $X$ has Li-Yorke chaos if, and only if, it has a scrambled pair.
\end{prop}
\begin{proof}
Let $x,y$ be a scrambled pair and, without loss of generality, assume that $x = \gamma_{1}\gamma_{2}\gamma_{3}\ldots$ is an infinite path. If $y$ is also an infinite path, then the result follows from the previous theorem. So, assume that $y$ is a finite path. Again without loss of generality we can assume that $y = v$, where $v$ is vertex which is a minimal infinite emitter. By Proposition~ \ref{teoinfemm} there is a subsequence $\{\gamma_{i_{k}}\}_{k \in \mathbb{N}}$ of distinct edges such that $s(\gamma_{i_{k}}) = v$ for all natural $k$. Now, for each natural $\ell$, consider the path $c_{\ell}=\gamma_{i_{\ell}}\gamma_{i_{\ell}+1}\ldots\gamma_{i_{(\ell+1)}-1}$. Proceeding as in Lemma~\ref{lemcpse}, we construct a closed path based at $v$, say $c'_{\ell}$, that begins at $\gamma_{i_{\ell}}$ for each natural $\ell$. As $\gamma_{i_{\ell}} \neq \gamma_{i_{\ell'}}$ for all $\ell \neq \ell'$, the set $\{c'_{\ell}: \ell \in \mathbb{N}\}$ has infinite cardinality. But $\{c'_{\ell}: \ell \in \mathbb{N}\} \subset CP_{\mathcal{G}}(v)$. Hence $\#CP_{\mathcal{G}}(v) \geq 2$ and by Theorem~\ref{teothebest} we get that $X$ has Li-Yorke chaos.
\end{proof}

Before we give an example of an ultragraph with a scrambled pair that does not present Li-Yorke chaos (Example~\ref{exedoesnot}), we summarize our results about chaoticity on ultragraphs shift spaces below.

\begin{teo}\label{teosumma}
Let $\mathcal{G}$ be an ultragraph and $X$ be the associated shif space. The following statements are equivalent:
\begin{enumerate}
    \item $\mathcal{G}$ has a vertex $v$ such that $\#CP_{\mathcal{G}}(v)\geq 2$;
    \item $X$ has a scrambled pair formed by infinite paths;
    \item $X$ has an uncountable scrambled set which is perfect and compact;
    \item  $X$ is Li-Yorke chaotic.
\end{enumerate}
\end{teo}
\begin{proof}
 Implication $1. \Rightarrow 3.$ is Theorem~\ref{teothebest} and $3. \Rightarrow 4.$ follows by Definition \ref{definscramble}. The statements 2. and 4. are equivalent by Theorem~ \ref{teopair} and, in the proof of Theorem~ \ref{teopair}, we showed that $2. \Rightarrow 1.$.
\end{proof}

\begin{exe}\label{exedoesnot}
\textit{An ultragraph shift space which has a scrambled pair but does not present Li-Yorke chaos.} 

Let $\mathcal{G}$ be the ultragraph with vertices $\{u_{k}: k\geq 0\}\cup \{v_{k}: k\geq 0\}$, edges $\{e_{k}:k\geq 0\} \cup \{f_{k}:k\geq 0\}$, source map given by $s(e_{k}) = u_{k}$ and $s(f_{k}) = v_{k}$ for all $k \geq 0$, and range map given by:
\begin{equation*}
r(e_{k}) = \begin{cases}
           \{u_{2k}: k \geq 1\}  & \text{if } k=0\\
             u_{k+1}  & \text{if } k \neq 0 ,
       \end{cases} \quad
\end{equation*}

\begin{equation*}
r(f_{k}) = \begin{cases}
           \{v_{1}, v_{2}, u_{1}\}  & \text{if } k=0\\
             \{v_{2k+1}, v_{2k+2}\}& \text{if } k \neq 0.
       \end{cases} \quad
\end{equation*}
We represent $\G$ in the picture below:

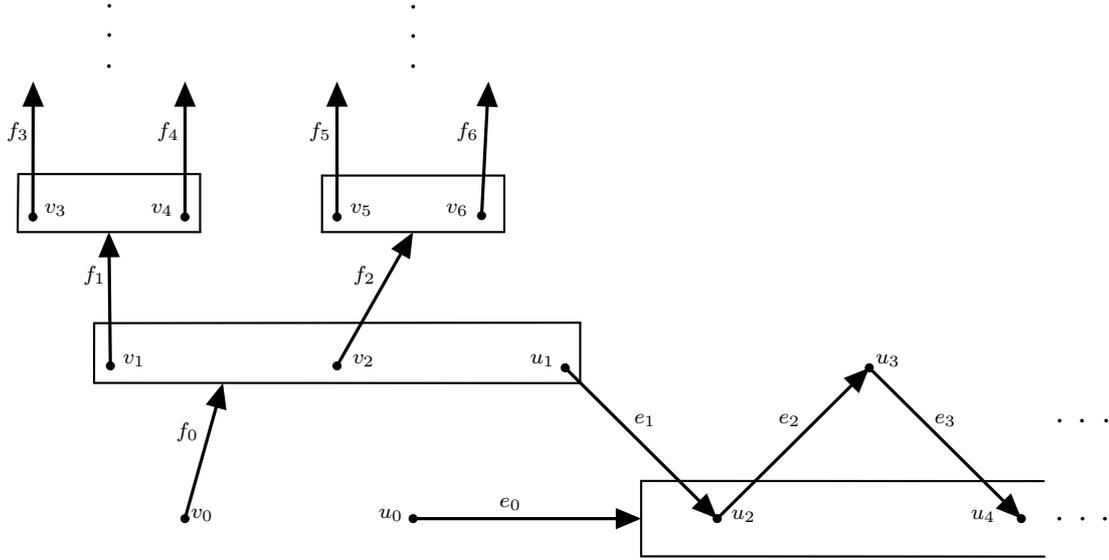
\begin{tikzpicture}[line cap=round,line join=round,>=triangle 45,x=1.0cm,y=1.0cm]
\clip(-5.,-1.) rectangle (11.,7.);
\draw [->,line width=1.2pt] (1.,0.) -- (4.,0.);
\draw [->,line width=1.2pt] (3.,2.) -- (5.,0.);
\draw [->,line width=1.2pt] (5.,0.) -- (7.,2.);
\draw [->,line width=1.2pt] (7.,2.) -- (9.,0.);
\draw [->,line width=1.2pt] (-2.,0.) -- (-1.5013065843621511,1.7964814814814956);
\draw [->,line width=1.2pt] (-2.980730289372258,2.0255242014501245) -- (-3.,3.8);
\draw [->,line width=1.2pt] (0.,2.0255242014501245) -- (0.9994341563785912,3.799444444444468);
\draw [->,line width=1.2pt] (-4.,4.) -- (-3.996121399176968,5.796481481481526);
\draw [->,line width=1.2pt] (-2.,4.) -- (-1.9990843621399292,5.796481481481526);
\draw [->,line width=1.2pt] (0.,4.) -- (0.,5.796481481481526);
\draw [->,line width=1.2pt] (1.9018714481677512,4.018012280357944) -- (1.9922932588673248,5.796481481481525);
\draw [line width=0.8pt] (-4.2,3.8)-- (-1.8,3.8);
\draw [line width=0.8pt] (-0.2,3.8)-- (2.2,3.8);
\draw [line width=0.8pt] (-3.195030864197519,1.7996502057613117)-- (3.20265946502057,1.7973405349794298);
\draw [line width=0.8pt] (-3.195030864197519,1.7996502057613117)-- (-3.1941341694428114,2.5924456202233936);
\draw [line width=0.8pt] (3.197919530601845,1.7973422461732311)-- (3.1979227905153804,2.598193872885227);
\draw [line width=0.8pt] (-3.1941341694428114,2.5924456202233936)-- (3.1979227905153804,2.598193872885227);
\draw [line width=0.8pt] (-4.2,3.8)-- (-4.192174537853551,4.569126004311182);
\draw [line width=0.8pt] (-1.8009014305310602,3.8)-- (-1.8009014305310602,4.569126004311182);
\draw [line width=0.8pt] (-4.192174537853551,4.569126004311182)-- (-1.8009014305310602,4.569126004311182);
\draw [line width=0.8pt] (-0.1971389378796784,3.8)-- (-0.1971389378796784,4.549725651577499);
\draw [line width=0.8pt] (2.2,3.8)-- (2.1998824221046447,4.549725651577499);
\draw [line width=0.8pt] (-0.1971389378796784,4.549725651577499)-- (2.1998824221046447,4.549725651577499);
\draw [line width=0.8pt] (4.,-0.5)-- (4.,0.5);
\draw [line width=0.8pt] (4.,0.5)-- (9.296279966033083,0.497207524985293);
\draw [line width=0.8pt] (4.,-0.5)-- (9.304902345025832,-0.5029884381736289);
\begin{scriptsize}
\draw [fill=black] (1.,0.) circle (1.5pt);
\draw[color=black] (0.7094682866287849,0.04668822261413716) node {$u_{0}$};
\draw [fill=black] (3.,2.) circle (1.5pt);
\draw[color=black] (2.7,2.1) node {$u_{1}$};
\draw [fill=black] (5.,0.) circle (1.5pt);
\draw[color=black] (5.3461525899797655,0.03) node {$u_{2}$};
\draw [fill=black] (7.,2.) circle (1.5pt);
\draw[color=black] (7.247387157881005,2.1) node {$u_{3}$};
\draw [fill=black] (9.,0.) circle (1.5pt);
\draw[color=black] (8.5,0.03) node {$u_{4}$};
\draw[color=black] (2.28089685805736,0.2) node {$e_{0}$};
\draw[color=black] (4.05,1.3077111503037273) node {$e_{1}$};
\draw[color=black] (5.95,1.3077111503037273) node {$e_{2}$};
\draw[color=black] (8.0,1.3077111503037273) node {$e_{3}$};
\draw [fill=black] (-2.,0.) circle (1.5pt);
\draw[color=black] (-1.7737768632830373,0.04668822261413716) node {$v_{0}$};
\draw [fill=black] (0.,2.0255242014501245) circle (1.5pt);
\draw[color=black] (0.35,2.1) node {$v_{2}$};
\draw [fill=black] (-2.980730289372258,2.0255242014501245) circle (1.5pt);
\draw[color=black] (-2.666193089032598,2.1) node {$v_{1}$};
\draw [fill=black] (-4.,4.) circle (1.5pt);
\draw[color=black] (-3.7,4.1) node {$v_{3}$};
\draw [fill=black] (-2.,4.) circle (1.5pt);
\draw[color=black] (-2.3,4.1) node {$v_{4}$};
\draw [fill=black] (0.,4.) circle (1.5pt);
\draw[color=black] (0.3214612319550628,4.1) node {$v_{5}$};
\draw [fill=black] (1.9018714481677512,4.018012280357944) circle (1.5pt);
\draw[color=black] (1.58248415964466,4.1) node {$v_{6}$};
\draw[color=black] (-1.9677803906198983,1.1525083284342394) node {$f_{0}$};
\draw[color=black] (-3.190002612842123,3.228346070938642) node {$f_{1}$};
\draw[color=black] (0.36,3.228346070938642) node {$f_{2}$};
\draw[color=black] (-4.198820954993801,5.12958063883987) node {$f_{3}$};
\draw[color=black] (-2.2199849761578174,5.12958063883987) node {$f_{4}$};
\draw[color=black] (-0.2217486445881483,5.12958063883987) node {$f_{5}$};
\draw[color=black] (1.7182866287804628,5.12958063883987) node {$f_{6}$};
\draw [fill=black] (-3.,6.) circle (0.5pt);
\draw [fill=black] (-2.9949212881311724,6.412159514011333) circle (0.5pt);
\draw [fill=black] (-2.9949212881311724,6.800166568685053) circle (0.5pt);
\draw [fill=black] (1.,6.) circle (0.5pt);
\draw [fill=black] (1.0015513750081704,6.399225945522209) circle (0.5pt);
\draw [fill=black] (1.0015513750081704,6.787233000195928) circle (0.5pt);
\draw [fill=black] (9.503217061859061,1.2990887713109789) circle (0.5pt);
\draw [fill=black] (9.79637794761254,1.2990887713109789) circle (0.5pt);
\draw [fill=black] (10.115405970344266,1.2990887713109789) circle (0.5pt);
\draw [fill=black] (9.5,0.) circle (0.5pt);
\draw [fill=black] (9.80500032660529,0.) circle (0.5pt);
\draw [fill=black] (10.098161212358768,0.) circle (0.5pt);
\end{scriptsize}
\end{tikzpicture}

Let $X$ be the shift space associated to $\mathcal{G}$, let $x = e_{0}e_{2}e_{3}e_{4}\ldots$ and let $y = r(e_{0})$. 
As $s(e_{2k-1}) \notin r(e_{0})$ for all $k \geq 1$ we have, by Proposition~\ref{teosupemm}, that $\displaystyle{\limsup_{n \rightarrow \infty} d_{X}(\sigma^{n}(x),y) > 0}$. Since $s(e_{2k}) \in r(e_{0})$ for all $k \geq 1$ we have, by Proposition~\ref{teoinfemm}, that $\displaystyle{\liminf_{n \rightarrow \infty} d_{X}(\sigma^{n}(x),y) = 0}$. Hence $x,y$ is a scrambled pair. 

Now, as $\G$ has no vertex $v$ such that $\#CP_{\mathcal{G}}(v)\geq 2$, it follows by Teorema~\ref{teosumma} that $X$ does not have an uncountable scrambled set. In fact, it is not hard to check that the only scrambled pairs in $X$ different from $(x,y)$ are of the form 
 $(\sigma^{n}(f_0e_1e_2e_3,\ldots),y)$ for $n\geq 0$.
\end{exe}

\section*{Acknowledgments}

\noindent D. Gon\c{c}alves was partially supported by Conselho Nacional de Desenvolvimento Cient\'{\i}fico e Tecnol\'{o}gico -
CNPq.

\noindent B. B. Uggioni was partially supported by Conselho Nacional de Desenvolvimento Cient\'{\i}fico e Tecnol\'{o}gico -
CNPq.

\end{document}